\documentclass[11pt,a4paper]{amsart}
\pdfoutput=1
\usepackage[english]{babel}
\usepackage[T1]{fontenc}
\usepackage[utf8]{inputenc}
\usepackage{lmodern}
\usepackage{amsmath,amsfonts,amssymb}
\usepackage[pdftex]{hyperref}

\title{The geometry of K\"ahler cones}
\author{Gunnar Þór Magnússon}
\address{Institut Fourier\\
100 rue des Maths\\
38402 St. Martin d'Hères\\
France}
\email{gunnar.magnusson@ujf-grenoble.fr}

\newtheorem{theo}{Theorem}[section]
\newtheorem{prop}[theo]{Proposition}
\newtheorem{coro}[theo]{Corollary}
\newtheorem{lemm}[theo]{Lemma}
\newtheorem*{theo*}{Theorem}
\theoremstyle{definition}

\newtheorem{exam}[theo]{Example}
\theoremstyle{remark}
\newtheorem*{rema}{Remark}

\newcommand{\RR}{\mathbb{R}}
\newcommand{\CC}{\mathbb{C}}

\newcommand{\cc}[1]{\mathcal{#1}}

\newcommand{\End}{\mathop{\mathrm{End}}}
\newcommand{\Aut}{\mathop{\mathrm{Aut}}}
\newcommand{\id}{\mathop{\mathrm{id}}}
\renewcommand{\d}{\text{d}}
\newcommand{\Vol}{\mathop{\mathrm{Vol}}}
\newcommand{\tr}{\mathrm{tr}}
\newcommand{\RRic}{\mathop{\mathrm{Ric}}}
\def\Im{\mathop{\rm Im}}

\def\qandq{\quad \text{and} \quad}
\def\dV{\text d V}
\def\II{\mathop{\rm{I}\mkern -2mu \rm{I}}}
\def\KC{\mathcal{K}}
\def\NKC{\widehat{\KC}}
\def\KCC{\widehat{\KC_\CC}}
\def\KCComp{\KC_\CC}

\begin{document}

\begin{abstract}
  The K\"ahler cone of a compact manifold carries a natural Riemannian metric,
given by the intersection product of its cohomology ring. We write down
the curvature tensor of this metric by embedding the K\"ahler cone in the
space of hermitian metrics on the underlying manifold. After discussing weak functorality and completeness properties, we give a relative version of both the K\"ahler cone and the metric.
\end{abstract}

\maketitle

\section*{Introduction}

Let $X$ be a compact K\"ahler manifold and let $\KC$ be the K\"ahler cone of $X$. The smooth function $\omega \mapsto -\log \Vol(X,\omega)$ is strictly convex on $\KC$ and thus defines a Riemannian metric $g$ on the K\"ahler cone. The level sets $\KC_\lambda \subset \KC$ of K\"ahler classes of volume $\lambda$ and the restriction of this metric to those sets have been studied by Huybrechts, Wilson and Trenner \cite{Huybrechts,Wilson,WilsonTrenner}.

Wilson, alone at first and later with Trenner, obtained explicit formulas for the curvature tensor of the metric $g$, expressed in terms of the
intersection product on the cohomology ring of $X$. He postulated that this metric should correspond to the Weil--Petersson metric on the space of complex moduli under mirror symmetry and asked whether the sectional curvature of $g$ is negative.

We investigate this question of negativity in this paper. We use the Aubin--Calabi--Yau theorem \cite{Aubin,Yau} to embed the K\"ahler cone $\cc K$ into the space $\cc M$ of Hermitian metrics on $X$. This latter infinite-dimensional space is equipped with the normalized Hodge $L^2$ metric
$$
G(U,V)_\Omega = \frac{1}{\Vol(X,\Omega)} 
\int_X \langle U,V \rangle \,dV_\Omega,
$$
where the inner product is the one the hermitian metric $\Omega$ induces on  $(1,1)$-forms. Our main result, described in Theorem~\ref{theo:thth} and Proposition~\ref{prop:thfo}, can be summarized as follows:

\begin{theo*}
{\rm (1)} The curvature tensor of $(\cc M,G)$ is
$$
    R(U,V,Z,W) = \tfrac 14 G(\{Z,W\},\{U,V\})
$$
and the sectional curvature of $\cc M$ is nonpositive.

\smallskip\noindent
{\rm (2)} The embedding $(K,g) \hookrightarrow (\cc M,G)$ given by the Aubin--Calabi--Yau theorem is Riemannian. Its second fundamental form is $    \II(U,V) = \Delta Gr \, \nabla_V U$, where $\Delta$ and $Gr$ are the Laplacian and the Green operator associated to a given hermitian metric, and $\nabla$ is the Levi-Civita connection of $G$.
\end{theo*}

The Levi-Civita connection of $G$ is described in Proposition~\ref{levicitiva}. Together, these results describe the curvature tensor of $(K,g)$ in an analytic manner. Unfortunately we are not able to improve on Wilson's estimates on the negativity of its sectional curvatures at this time, but our results hopefully open the way for an analytic approach to the problem. 

We had more luck investigating the completeness of the metric $g$. Note that the potential $-\log \Vol$ of the metric $g$ is well-defined on the volume cone $\{ \alpha \in H^{1,1}(X,\RR) \mid \alpha^n > 0\}$. The K\"ahler cone is of course contained in this cone, but is in almost all cases smaller than it.

\begingroup
\def\thesection{4}
\def\thetheo{\ref{prop:fofo}}
\begin{prop}
  The metric on the K\"ahler cone is complete if and only if the K\"ahler cone is a connected component of the volume cone.
\end{prop}
\endgroup

The paper is organized as follows. We start by reviewing the construction of the Riemannian metric on the K\"ahler cone of a given manifold in Section~\ref{seon}. Next we exhibit some examples of this metric; these includes the remark that the simple structure of the cohomology ring of a compact surface allows for a complete answer to Wilson's question in that case. We then prove our main results in Sections~\ref{setw} and~\ref{seth}. In Section~\ref{sefo} we discuss weak functorality properties of the metric under pullbacks and its completeness. We then finish the paper by introducing a relative version of the K\"ahler cone that varies along with families of K\"ahler manifolds, and show that the metric on a single cone extends to a smooth closed $(1,1)$-form on the total space of the relative cone.

\section{The K\"ahler cone}
\label{seon}

Let $X$ be a compact K\"ahler manifold of dimension $\dim_\CC X = n$.
Let
\begin{equation*}
  K :=
  \{ \omega \in H^{1,1}(X,\RR)
  \mid \omega \text{ contains a K\"ahler metric}\}
\end{equation*}
be the K\"ahler cone of $X$. The set $K$ is an open cone
in the finite-dimensional real vector space $H^{1,1}(X,\RR)$. It is the trancendental analogue of the ample cone of a projective variety.

The function $\Vol : \KC \to \RR^*_+$ that sends a K\"ahler class $\omega$ to the volume of $X$ with respect to the class is smooth. It is also a surjective submersion and its fiber over a point $\lambda$ is the set $\KC_\lambda$ of K\"ahler classes of volume $\lambda$. Let us define a symmetric bilinear form $g$ on the tangent space of $\KC$ by
$$
\displaylines{
g(u,v) = -D_u D_v \log \Vol 
\hfill\cr
\phantom{g(u,v)}
{}= \frac{1}{\Vol(X,\omega)} \int_X u \wedge
  \frac{\omega^{n-1}}{(n-1)!} \,
  \frac{1}{\Vol(X,\omega)} \int_X v \wedge
  \frac{\omega^{n-1}}{(n-1)!}
\hfill\cr\hfill
  {}- \frac{1}{\Vol(X,\omega)}
  \int_X u \wedge v \wedge \frac{\omega^{n-2}}{(n-2)!}.
}
$$
Here $D$ is the flat connection on $\KC$ defined by the exterior derivative on the finite-dimensional vector space $H^{1,1}(X,\RR)$. We denote it thus to avoid confusion with the exterior derivative $d$ on the manifold $X$. 

\begin{rema}
  The tangent space of a level set $\KC_\lambda$ is the space
$$
T_{\KC_\lambda,\omega} 
= \{ u \in H^{1,1}(X,\RR) \mid u \wedge \omega^{n-1} = 0\}
$$
of $\omega$-primitive classes; see \cite{Huybrechts}. The restriction of the bilinear form $g$ to the submanifold $\KC_\lambda$ is then
$$
g_\lambda(u,v) = {}- \frac{1}{\lambda}
  \int_X u \wedge v \wedge \frac{\omega^{n-2}}{(n-2)!}.
$$
This form is positive-definite by the hard Lefschetz theorem, and is the metric considered in \cite{Wilson}. Working with the whole K\"ahler cone instead of a level set $\KC_\lambda$ is largely a matter of taste and does not change much for the results proved, except that the former lets us avoid some gymnastics in the space of Hermitian metrics later on.
\end{rema}

\begin{prop}
\label{proponon}
  The bilinear form $g$ is a Riemannian metric on $\KC$. If $\Omega$ is a K\"ahler metric in a class $\omega$ and $U$ and $V$ are the harmonic representatives of $u$ and $v$, then
$$
    g(u,v)(\omega) =
    \frac{1}{\Vol(X,\Omega)} \int_X \langle U,V \rangle
    \dV_\Omega,
$$
where the inner product is the one induced by $\Omega$ on smooth
$(1,1)$-forms.
\end{prop}

\begin{proof}
  Let $u$ and $v$ be real $(1,1)$-classes and let $\omega$ be a K\"ahler class. We write $u = u_0 \omega + u_1$ and $v = v_0 \omega + v_1$ for the primitive decompositions of $u$ and $v$. Then
$$
g(u,v)_\omega = n^2 u_0v_0 + g_{\Vol(X,\omega)} (u_1,v_1),
$$
which shows that $g$ is positive-definite on the tangent space of $\KC$.

For the second part of the proposition, simply decompose the forms $U$ and $V$ into their primitive components. Using the hard Lefschetz theorem we find that the $L^2$ inner product of $U$ and $V$, once normalized by the volume $\Vol(X,\Omega)$, only depends on the cohomology classes of $U$ and $V$ and coincides with $g(u,v)$.
\end{proof}

We note that we can complexify the K\"ahler cone of $X$ by setting
$$
\KC_\CC = \{ a \in H^{1,1}(X,\CC) \mid \hbox{$\Im a$ is K\"ahler}\}.
$$
The function $-\log \Vol$ is then a global potential of a K\"ahler metric on $\KC_\CC$, by the same reasoning as before.

\begin{exam}
  Let $X$ be a compact K\"ahler manifold with Hodge number $h^{1,1}(X) = 1$.
Then the K\"ahler cone of $X$ is isomorphic to the positive real line. Let be $\omega_1$ the unique class in $\KC$ of volume one.
As $\Vol(X,t\omega_1) = t^n$ then the metric $g$ is defined by the Hessian of
$-n \log t$. The reader may be more familiar with the complexification
of this metric, which is just the Poincaré metric on the upper half-plane.
\end{exam}

\begin{exam}
\label{exam:onth}
  Let $X$ be a compact K\"ahler surface. Fix a K\"ahler class $\omega_0$ of volume $2$ on $X$ and let $u_1, \dots, u_N$ be $\omega_0$-primitive classes such that $(\omega_0,u_1,\dots,u_N)$ is a basis of $H^{1,1}(X,\RR)$ and such that $- \frac 12\int_X u_j \wedge u_k = 2\delta_{jk}$. This defines an isomorphism $\RR^{N+1} \to H^{1,1}(X,\RR)$. If $t \in \RR^{N+1}$, then the volume of the associated $(1,1)$-class on $X$ under this isomorphism is
$$
\Vol(X,t) = \frac{1}{2} \biggl( t_0 \omega_0 + \sum_{j=1}^N t_j u_j \biggr)^2
\!\!\! 
= t_0^2 - \sum_{j=1}^N t_j =: q(t),
$$
where $q$ is the standard quadratic form of signature $(1,N)$ on $\RR^{N+1}$. We will write $\cc P$ for the connected component of the positive cone $\{t \mid q(t) > 0\}$ that contains the image of the K\"ahler cone. The metric $g$ under this isomorphism is given by the Hessian of $-\log q$, the Hessian being defined on all of $\cc P$.

The group $O^+(1,N)$ acts transitively on the level sets of $q$, so $\RR^*_+ \times O^+(1,N)$ acts transitively on $\cc P$ and preserves the Hessian of $-\log q$. Since the metric is positive-definite on the K\"ahler cone, it actually extends to a Riemannian metric on the whole of $\cc P$, which is then a complete homogeneous manifold of nonpositive sectional curvature.

This example implies that the metric on the K\"ahler cone of a surface is complete if and only if the K\"ahler cone is a connected component of the cone of classes of positive volume. This is actually true in general; see Proposition~\ref{prop:fofo}.
\end{exam}

\begin{exam}
  Let $V$ be a complex vector space of dimension $n$ and let $\Gamma$ be
a lattice in $V$. Then $X = V/\Gamma$ is a complex torus. Its degree $(1,1)$
cohomology group is canonically isomorphic to $\bigwedge^{1,1}V^*$. If we pick a basis of $V$, then an element $\omega$ in $\bigwedge^{1,1}V^*$ is a $n \times n$ matrix $\Omega$ of complex numbers. The element $\omega$ is real if $\Omega$ is hermitian, and a K\"ahler class if $\Omega$ is positive-definite. One may calculate that the metric on $\KC$ is
$$
  g(U,V)(\Omega) = 
  \tr (\Omega^{-1} U \Omega^{-1} V).
$$
This is the well-known Maass metric on the space of Hermitian matrices \cite{Maass}. Interestingly, the Weil--Petersson metric on the space of polarized abelian varieties is given by an expression very similar to this one \cite{Schumacher}.
\end{exam}

\section{The Aubin--Calabi--Yau theorem}
\label{setw}

\begin{theo}[\cite{Aubin,Yau}]
\label{aubin--calabi--yau}
  Let $X$ be a compact K\"ahler manifold. If $\dV$ is a smooth volume
  form\footnote{To be completely precise we need $\dV$ to be
  compatible with the orientation defined by the complex structure
  on $X$, that is, we want $\Vol(X,\dV) > 0$.} on $X$, then every K\"ahler
  class $\omega$ contains a unique K\"ahler metric $\Omega$ whose
  volume form is
  \begin{align*}
    \dV_\Omega = \frac{\Omega^n}{n!} = c \, \dV,
  \end{align*}
where the constant $c$ is $\Vol(X,\Omega)/\!\Vol(X,\dV)$.
\end{theo}

The reader may recall that the Aubin--Calabi--Yau theorem is usually stated as saying that if a smooth form $\rho$ represents the class $2\pi c_1(X)$, then every K\"ahler class contains a unique metric $\Omega$ whose Ricci-form is $\rho$. Choosing a form $\rho$ results in the same metrics in each class as choosing a volume form $\dV$.

Let $\cc M$ be the space of all hermitian metrics $\Omega$ on $X$. It is an infinite dimensional manifold that has the structure of an open set in the vector space of smooth $(1,1)$-forms on $X$. The space $\cc M$ is equipped with a Riemannian metric
\begin{align*}
  G(U,V)(\Omega) =
  \frac{1}{\Vol(X,\Omega)} \int_X \langle U,V \rangle \, \dV_\Omega,
\end{align*}
where the inner product under the integral sign is the one induced by $\Omega$ on the space of smooth $(1,1)$-forms on $X$. The unnormalized version of this metric is known as the Ebin metric \cite{Ebin} and has received much attention in the Riemannian world, see for example \cite{ClarkeRubinstein}.

Now, and for the rest of the paper, we fix a volume form $\dV$ that is compatible with the orientation defined by the complex structure of $X$.
Let $\cc M_K \subset \cc M$ be the closed subspace of K\"ahler metrics on $X$.
It is a smooth submanifold of $\cc M$.  Following Huybrechts \cite{Huybrechts} we define the \em nonlinear K\"ahler cone \em of $X$ by
$$
  \NKC = \{ \Omega \in \cc M_K \mid \dV_\Omega = c \, \dV, \  c > 0\}
  = \{ \Omega \in \cc M_K \mid \RRic \Omega = \rho \}
$$
where $\rho$ is the curvature form of the hermitian metric defined by $\dV$ on the canonical bundle of $X$. Note that there is a smooth map $p : \cc M_K \to K$, given by sending a K\"ahler metric to its cohomology class. The Aubin--Calabi--Yau theorem now says that the restriction of $p$ to $\NKC$ is a bijection. We refer to \cite[Section~1]{Huybrechts} for the proof of:

\begin{prop}
  The set $\NKC$ is a smooth submanifold of $\cc M_K$, whose tangent space at $\Omega$ is the space of $\Omega$-harmonic $(1,1)$-forms on $X$. The smooth map $p : \NKC \to K$ is a diffeomorphism.
\end{prop}

Denote by $f : K \to \NKC \hookrightarrow \cc M$ the composition of inverse of the diffeomorphism $p$ and the injection of $\NKC$ into $\cc M$. By the above, it is an embedding of the K\"ahler cone $K$ into the space $\cc M$ of hermitian metrics on $X$.

\begin{prop}
  The morphism $f : K \to \cc M$ is an isometric embedding of Riemannian manifolds.
\end{prop}

\begin{proof}
  Let $\omega$ be a point of $K$ and denote by $\Omega$ its image under $f$. One easily checks that if $u$ is a tangent vector of $K$ at $\omega$, then $f_*u$ is a $\Omega$-harmonic form on $X$ that represents the class $u$. The pullback of $G$ to $K$ is now
\begin{align*}
  f^*G(u,v) = G(f_*u, f_*v),
\end{align*}
but the right hand side here is equal to $g(u,v)$ by Proposition~\ref{proponon}.
\end{proof}

\section{The curvature tensors}
\label{seth}

The curvature tensor of $\cc M$ seems to be known, it is basically the curvature tensor of a locally symmetric space of noncompact type. However I had a hard time finding a suitable reference for this fact, so we will calculate this tensor here. We refer to \cite{Lang} for background on differential calculus on infinite-dimensional manifolds. We'll denote the exterior derivative on $M$ by $D$ to avoid confusion with the exterior derivative $\d$ on $X$.

Let's fix some notation. The space of smooth $(p,q)$-forms on $X$ will be denoted by $\cc A^{p,q}$. We note that the tangent bundle $T_{\cc M}$ is the trivial bundle with fiber $\cc A^{1,1}$, so the exterior derivative on $\cc M$ defines a flat connection on $\cc M$. Remark that we possess a smooth vector bundle over the manifold $\cc M$. If we denote it by $\cc H$, then its fiber of a point $\Omega$ is
\begin{equation*}
  \cc H_\Omega = \cc H^{1,1}(\Omega),
\end{equation*}
the space of $\Omega$-harmonic $(1,1)$-forms on $X$. The tangent bundle of $\NKC$ is just the restriction of $\cc H$ to the space $\NKC$. Hodge theory shows that the quotient bundle of $\cc H$ in $T_{\cc M}$ identifies with the bundle whose fibers consists of the forms that are either $\d$ or $\d^*$-exact.

Recall that if $T$ is a complex vector space of dimension $n$, then a $(1,1)$-form $u$ on $T$ may be viewed as a linear morphism $u : T \to \overline T^*$. If $\omega$ is a hermitian inner product on $T$, then the inner product $\omega$ induces on $(1,1)$-forms is
$$
  \langle u, \overline v \rangle
  = \tr(\omega^{-1} u \, \omega^{-1} {}^t\overline v).
$$

The reader may enjoy comparing the following expression of
the Levi-Civita connection with the one given in Section~3 of
\cite{ClarkeRubinstein}.

\begin{prop}
\label{levicitiva}
  Let $U$ and $Z$ be tangent fields on a neighborhood of a point
  $\Omega_0$. The Levi-Civita connection is 
  \begin{align*}
    \nabla_{\!Z} \, U &=
    \tfrac{1}{2} \bigl(
    \langle Z, \Omega \rangle - G(Z,\Omega)
    \bigr) \, U
    -\tfrac{1}{2}
    \bigl( Z \Omega^{-1} U + U \Omega^{-1} Z \bigr)
    + D_Z U \\
    &=: T(Z) \, U + S(Z,U) + D_Z U.
  \end{align*}
In particular, if $\Omega$ is K\"ahler and $Z$ is $\Omega$-harmonic,
then
\begin{align*}
  \nabla_{\!Z} \, U =
    -\tfrac{1}{2} \bigl( Z \Omega^{-1} U + U \Omega^{-1} Z \bigr)
    + D_Z U.
\end{align*}
\end{prop}

\begin{proof}
  Let $V$ be another vector field. The metric $G$ is 
\begin{align*}
  G(U,V) = \frac{1}{\Vol(X,\Omega)}
  \int_X \!\tr(\Omega^{-1}U\Omega^{-1}V) \,\dV_\Omega
\end{align*}
and its Levi-Civita connection is the unique symmetric connection that satisfies
\begin{align*}
  Z \cdot G(U,V) = G(\nabla_{\!Z}U,V) + G(U,\nabla_{\!Z}V).
\end{align*}
To differentiate the function $G(U,V)$ in the direction of a vector
field $Z$ we must differentiate three terms: the volume form
$\dV_\Omega$, the volume $\Vol(X,\Omega)$ and the inner product
$\langle U, V \rangle$ inside the integral.

First consider the inner product. Regard the metric $\Omega$ as a linear morphism $T_X \to \overline T_X^*$. Since $D_Z \Omega = Z$ we get $D_Z \Omega^{-1} = - \Omega^{-1} Z \, \Omega^{-1}$ by using standard formulas for
the derivative of the inverse of a linear morphism. Differentiating and collecting terms in an eccentric way we find that
$$
\def\filler{\phantom{Z \cdot \langle U,V \rangle}}
\displaylines{
  Z \cdot \langle U,V \rangle
  = {}-\tfrac{1}{2}\bigl( \langle Z\Omega^{-1}U,V \rangle
  + \langle U,Z\Omega^{-1}V \rangle
  \bigr)  
\hfill\cr\hfill\qquad\qquad
  {}-\tfrac{1}{2}\bigl( \langle U\Omega^{-1}Z,V \rangle
  + \langle U,V\Omega^{-1}Z \rangle \bigr) 
\hfill\cr \hfill
  {}+ \langle D_Z U, V \rangle + \langle U, D_Z V \rangle
}
$$
on a neighborhood of $\Omega_0$. Here the entries in the first pair of parentheses come from the first trace, and similarly for the second pair. We have split them in this way so the symmetry condition of the Levi-Civita connection  will be satisfied. These terms give the tensor
$$-\tfrac{1}{2} \bigl( Z \Omega^{-1} U + U \Omega^{-1} Z \bigr) +
D_Z U = S(Z,U) + D_Z U.$$

Next recall that the volume form of a hermitian metric is
$\dV_\Omega = \Omega^n / n!$. Differentiating this in the direction
of $Z$ we get
\begin{align*}
  Z \cdot \dV_\Omega = Z \wedge \frac{\Omega^{n-1}}{(n-1)!}
  = \tr_\Omega(Z) \, \dV_\Omega
  = \langle Z, \Omega \rangle \, \dV_\Omega.
\end{align*}
The derivative of the volume is then
\begin{align*}
  Z \cdot \Vol(X,\Omega) = \int_X \tr_\Omega(Z) \, \dV_\Omega
  = G(Z,\Omega) \Vol(X,\Omega).
\end{align*}
Thus he contributions of the volume and the volume form to $Z \cdot
G(U,V)$ are
\begin{align*}
  \frac{1}{\Vol(X,\Omega)} \int_X \langle U, V \rangle
  \langle Z, \Omega \rangle \dV_\Omega
  - G(Z,\Omega) G(U,V).
\end{align*}
We split each factor in two, and incorporate one into $U$ and
the other into $V$ as before. This gives the tensor $T(Z) \,U$
announced in the proposition.

Now, if $\Omega$ is K\"ahler and $Z$ is harmonic, then the function $\langle Z,\Omega \rangle = \tr_\Omega(Z) = \Lambda Z$ is also harmonic. It is thus constant on $X$, so $G(Z,\Omega) = \langle Z, \Omega \rangle$, and the above term vanishes.
\end{proof}

Note that even if we take the forms $U$ and $Z$ to be harmonic, it is absolutely not clear that the form $\nabla_U V$ is closed and thus represents a vector tangent to the space of K\"ahler metrics. In fact, this almost never happens and poses a problem when we try to estimate the curvature of our metric.

Let us define an affine connection $\nabla'$ on $\cc M$ by setting
\begin{align*}
  \nabla'_Z U = S(Z,U) + D_Z U.
\end{align*}
It differs from the Levi-Civita connection $\nabla$ only by the
tensor $T$. We'll also write $R'$ for the curvature tensor of the
connection $\nabla'$. 

\begin{lemm}
  The curvature tensors $R$ and $R'$ are equal.
\end{lemm}

\begin{proof}[Sketch of proof.]
The proof is a series of formal calculations, so we only indicate its main steps. First one shows that if $U$, $Z$ and $W$ are tangent fields on $\cc M$, then
$$
\displaylines{
   \nabla_Z \nabla_W U
  = (\nabla_Z T(W)) U + T(W) T(Z) U 
  \hfill\cr\hfill
  + T(W) \nabla'_Z U
    + T(Z) \nabla'_W U + \nabla'_Z \nabla'_W U,
}
$$
where $T(Z) = \frac{1}{2} (\langle Z, \Omega \rangle - G(Z,\Omega))$.
A formal substitution gives a similar formula for $\nabla_W \nabla_Z U$. A calculation shows that
$$
\nabla_{\!Z} T(W) - \nabla_W T(Z) = T([Z,W]).
$$
This last step permits us to compare the tensors $R(Z,W)U$ and $R'(Z,W)U$, which turn out to be equal.
\end{proof}

Some notation will be useful before going further. If $Z$ and $W$
are $(1,1)$-forms, we set
\begin{align*}
  \{ Z, W \} := Z \Omega^{-1} W - W \Omega^{-1} Z.
\end{align*}
This is again a $(1,1)$-form, and real if $Z$ and $W$ are real. This
bracket is antisymmetric and satisfies the Jacobi identity, as the
reader may find pleasure in verifying.\footnote{Just note that this
is the commutator on the space of global sections of $\End T_X$
under the isometry $\Omega : \End T_X \to \bigwedge^{1,1}T_X^*$.}

\begin{theo}
  \label{theo:thth}
  The curvature tensor of $\cc M$ is
  \begin{align*}
    R(U,V,Z,W) = \tfrac 14 G(\{Z,W\},\{U,V\})
  \end{align*}
and the sectional curvature of $\cc M$ is nonpositive.
\end{theo}

\begin{proof}
  It is enough to show that the identity holds for the curvature
  tensor $R'$.
We start by noting that
\begin{align*}
  \nabla&'_Z \nabla'_W U =
  -\tfrac 12 \nabla'_Z (W\Omega^{-1}U + U\Omega^{-1}W)
  + \nabla'_Z D_W U \\
  &= \tfrac{1}{4}
  \bigl(
  Z \Omega^{-1} W \Omega^{-1} U
  + Z \Omega^{-1} U \Omega^{-1} W
  {}+ W \Omega^{-1} U \Omega^{-1} Z
  + U \Omega^{-1} W \Omega^{-1} Z
  \bigr) \\
  &{}-\tfrac{1}{2}
  D_Z (W\Omega^{-1}U + U\Omega^{-1}W)
  {}-\tfrac{1}{2} ( Z \Omega^{-1} D_WU + D_W U \Omega^{-1} Z)
  + D_Z D_W U.
\end{align*}
Next we see that
$$
\displaylines{
  D_Z (W \Omega^{-1} U + U \Omega^{-1} W) =
  D_ZW \Omega^{-1} U
  - W \Omega^{-1} Z \Omega^{-1} U
  + W \Omega^{-1} D_Z U
\hfill\cr\hfill
{}+ D_Z U \Omega^{-1} W
  - U \Omega^{-1} Z \Omega^{-1} W
  + U \Omega^{-1} D_Z W,
}
$$
so in total
\begin{align*}
  \nabla'_Z \nabla'_W U &=
  \tfrac{1}{4}
  \bigl(
  Z \Omega^{-1} W \Omega^{-1} U
  + U \Omega^{-1} W \Omega^{-1} Z
  \bigr) \\
  &{}+
  \tfrac{1}{4}
  \bigl(
  Z \Omega^{-1} U \Omega^{-1} W
  {}+ W \Omega^{-1} U \Omega^{-1} Z
  \bigr) \\
  &{}+
  \tfrac{1}{2}
  \bigl(
  W \Omega^{-1} Z \Omega^{-1} U
  +  U \Omega^{-1} Z \Omega^{-1} W
  \bigr)\\
  &{}-
  \tfrac{1}{2}
  \bigl(
  D_ZW \Omega^{-1} U
  + U \Omega^{-1} D_Z W
  \bigr) \\
  &{}-
  \tfrac{1}{2}
  \bigl(
  D_Z U \Omega^{-1} W
  + D_W U \Omega^{-1} Z
  \bigr) \\
  &{}-\tfrac{1}{2}
  \bigl(
  Z \Omega^{-1} D_WU
  + W \Omega^{-1} D_Z U
  \bigr) \\
  &{}+ D_Z D_W U.
\end{align*}
We encourage the reader to stare at this expression for a little
while, and to appreciate that we have moved some terms between
parentheses.

Remark that the term $\nabla'_W \nabla'_Z U$ can be obtained by
formally exchanging the fields $Z$ and $W$. Do so, and write the resulting
mess next to the above expression so we can compare them line
for line.
The first line of the difference between the two is
$$
\displaylines{
  \tfrac{1}{4}
  \bigl(
  Z \Omega^{-1} W \Omega^{-1} U
  + U \Omega^{-1} W \Omega^{-1} Z
  - W \Omega^{-1} Z \Omega^{-1} U
  - U \Omega^{-1} Z \Omega^{-1} W
  \bigr)
\hfill\cr\hfill
= \tfrac{1}{4} (\{Z,W\} \Omega^{-1} U + U \Omega^{-1} \{W,Z\})
  = \tfrac{1}{4} \{\{Z,W\}, U \}.
}
$$
We note that the second line of the expression is symmetric in $Z$
and $W$, so it contributes nothing to the curvature tensor. Now,
the third line of the difference is
$$
\displaylines{
  \tfrac{1}{2}
  \bigl(
  W \Omega^{-1} Z \Omega^{-1} U
  + U \Omega^{-1} Z \Omega^{-1} W
  - Z \Omega^{-1} W \Omega^{-1} U
  - U \Omega^{-1} W \Omega^{-1} Z
  \bigr)
\hfill\cr\hfill
= \tfrac{1}{2} (\{W,Z\} \Omega^{-1} U + U \Omega^{-1} \{Z,W\})
  = -\tfrac{1}{2} \{\{Z,W\}, U \}.
}
$$
The fourth and seventh lines together give
$$
\displaylines{
  {}-
  \tfrac{1}{2}
  \bigl(
  D_ZW \Omega^{-1} U
  + U \Omega^{-1} D_Z W
  - D_W Z \Omega^{-1} U
  - U \Omega^{-1} D_W Z
  \bigr)
  \hfill\cr\hfill
  + D_Z D_W U - D_W D_Z U
  = -\tfrac{1}{2}
  \bigl( [Z,W] \Omega^{-1} U + U \Omega^{-1} [Z,W] \bigr)
  + D_{[Z,W]} U
  \cr\hfill
  = \nabla'_{[Z,W]} U,
  \qquad\mkern7mu
  \hfill
}
$$
which looks very promising. This leaves the fifth and sixth
lines. But both of them are symmetric in $Z$ and $W$ and thus
contribute nothing to the curvature tensor. Taken together, we have
\begin{align*}
  \nabla'_Z \nabla'_W U - \nabla'_W \nabla'_Z U
  = -\tfrac{1}{4} \{\{Z,W\}, U \} +  \nabla'_{[Z,W]} U,
\end{align*}
which gives $R(Z,W) \,U = -\tfrac{1}{4} \{\{Z,W\}, U \}$.

By picking a hermitian metric $\Omega$ and an orthonormal frame at a given point $x$, it is easy to check that
\begin{align*}
  \langle \{\{Z,W\}, U \}, V \rangle
  = -\langle \{Z,W\}, \{U,V\} \rangle.
\end{align*}
This implies that the curvature tensor has the stated form. If the tangent fields $U$ and $V$ have unit norm, the sectional curvature
of the metric is
$$
  K(U,V) = R(U,V,V,U)
  = \tfrac{1}{4} G(\{U,V\}, \{V,U\})
  = -\tfrac{1}{4} G(\{U,V\}, \{U,V\}),
$$
which is nonpositive.
\end{proof}

Recall that the nonlinear K\"ahler cone $\NKC$ is the subspace of $\cc M$ defined by
\begin{align*}
  \NKC = \{ \Omega \in \cc M \mid d \Omega = 0
\qandq \RRic \Omega = \rho \}
\end{align*}
where $\rho$ is a fixed smooth $(1,1)$-form that represents the Chern class $-c_1(X)$. Huybrechts \cite[Section~1]{Huybrechts} showed that the tangent space of $\NKC$ at a point $\Omega$ is the space of real harmonic $(1,1)$-forms. We thus get a short exact sequence
\begin{align*}
  0 \longrightarrow T_{\NKC}
  \longrightarrow T_{\cc M | \NKC}
  \longrightarrow N_{\NKC / \cc M}
  \longrightarrow 0
\end{align*}
of vector bundles over $\NKC$. 

\begin{prop}
  \label{prop:thfo}
  The second fundamental form of $\NKC$ in $\cc M$ at a point $\Omega$ is
  \begin{align*}
    \II(U,V) = \Delta Gr \, \nabla_V U,
  \end{align*}
where $\Delta$ and $Gr$ are the Laplacian and the Green operator associated to the metric $\Omega$.
\end{prop}

\begin{proof}
  We can decompose the identity morphism on the space of smooth $(1,1)$-forms as
\begin{align*}
  \id = h_\Omega + \Delta Gr,
\end{align*}
where $h_\Omega$ is the projection onto the space of harmonic forms and $Gr$ is the Green operator. This decomposition is orthogonal by Hodge theory and the operator $h_\Omega$ identifies with the projection onto $T_{\NKC}$. The operator $\Delta Gr$ thus identifies with the orthogonal projection $pr$ onto the normal bundle $N_{\smash{\NKC/\cc M}}$. By definition, we then have  $\II(U,V) = pr(\nabla_U V)$.
\end{proof}

\begin{coro}
\label{curvaturetensor}
  The curvature tensor of the space $\NKC$ at a point $\Omega$ is
\begin{equation*}
\displaylines{
    R^{\NKC}(U,V,Z,W) =
    R^{\cc M}(U,V,Z,W) 
    \hfill \cr \hfill
    {}+ G(\II(U,W), \II(V,Z)) - G(\II(U,Z), \II(V,W)).
    \quad
    \square
  }
\end{equation*}
\end{coro}

We note that the space of K\"ahler metrics is not totally geodesic in the space of Hermitian metrics, so the second fundamental form really does contribute to the curvature tensor. 

It would have been nice to have been able to use the above formula to answer Wilson's question \cite{Wilson} on the sectional curvature of the metric on the K\"ahler cone. This does not seem to follow easily from this analytic formula. I tried applying Bochner--Weitzenböck-type formulas to this situation, but without success. Deeper analysis seems needed to extract information on the second fundamental form. A difficulty lies in the $(1,1)$-form $\nabla_V U$, which mixes the forms $U$, $V$ and $\Omega$ in a way that makes it hard to extract information from what we know about the three original forms.

\section{Finite morphisms and completeness}
\label{sefo}

Let $Y$ be another compact K\"ahler manifold and let $f : X \to Y$
be a holomorphic morphism. If $\omega$ is a K\"ahler class on $Y$ then its pullback $f^* \omega$ is not a K\"ahler class on $X$ in general.
However we can impose some conditions on $f$ that ensure
this is the case and thus get a well defined holomorphic morphism
of K\"ahler cones $f^* : \KC(Y) \to \KC(X)$. For example,
this is the case if $f$ is either a finite morphism or the inclusion
of a submanifold into $X$. We can say something about at least one
of those cases.

\begin{prop}
  Let $f : X \to Y$ be a finite surjective morphism. Let $g_X$
  and $g_Y$ be the Riemannian metrics on the K\"ahler cones of $X$
  and $Y$, respectively. Then the pullback morphism $f^* : \KC(Y)
  \to \KC(X)$ is a Riemannian embedding.
\end{prop}

\begin{proof}
Let $\omega$ be a point in $\KC(Y)$. The volume of $X$ with respect to $f^*\omega$ is
\begin{align*}
  \Vol(X,f^*\omega) = p \, \Vol(Y,\omega)
\end{align*}
as $f$ is finite of degree $p$. Taking logarithms and Hessians now shows that $f^*$ is an embedding.
\end{proof}

\begin{coro}
  The group $\Aut X$ of holomorphic automorphisms of $X$ acts by isometries on the K\"ahler cone $K(X)$.
\end{coro}

A closer look reveals that this last statement contains less information than first meets the eye. The automorphism group $\Aut X$ of a compact complex manifold is a Lie group and it splits roughly into two parts; a positive-dimensional group given by the flows of holomorphic vector fields, or elements of $H^0(X,T_X)$, and a discrete part consisting of ``other'' automorphisms. The isomorphisms generated by vector fields act trivially on the cohomology ring of $X$, so the only part of $\Aut X$ that possibly acts by nontrivial isometries on $K(X)$ is discrete.

The K\"ahler cone of a compact complex manifold $X$ is described by the following result:

\begin{theo}[Demailly--Paun \cite{DemaillyPaun}]
  Let $X$ be a compact K\"ahler manifold. Then the K\"ahler cone of $X$ is one of the connected components of the set of real $(1,1)$ cohomology classes $a$ that are numerically positive on analytic cycles, i.e., such that $\int_Z a^p > 0$ for every irreducible analytic set $Z$ in $X$ of dimension $p$.
\end{theo}

The boundary of the K\"ahler cone of a compact complex manifold then
consists of three parts: 
\begin{enumerate}
\item Limits of classes $a_t$ whose volume $\frac{1}{n!}\int_X a_t^n$ tends to zero.
\item Limits of classes whose volume tends to infinity.
\item Limits of classes whose volume tends to some positive real number, but there exists a proper irreducible complex subspace $Z \subset X$ of dimension $p \geq 1$ whose volume tends to zero.
\end{enumerate}

Let us conspire to call $\cc P := \{\alpha \in H^{1,1}(X,\RR) \mid \alpha^n > 0\}$ the cone of volume classes on $X$, or the volume cone. It contains the K\"ahler cone, but is in almost all cases bigger than it.

\begin{prop}
\label{prop:fofo}
  The metric on the K\"ahler cone of $X$ is complete if and only if the 
K\"ahler cone is a connected component of the volume cone.
\end{prop}

\begin{proof}
  We first show that the classes on the first two parts of the boundary pose no problems. Let $I$ be an interval in the real numbers and let $\gamma : I \to \KC$ be a smooth path in $\KC$ that approaches the boundary of $\KC$. Let $I_m = [a, b_m]$ be an increasing exhaustion of $I$ by compact intervals and let $\gamma_m$ be the restriction of $\gamma$ to $I_m$. Suppose that the volume $\Vol(X,\gamma_m)$ tends to either zero or infinity as $m$ tends to infinity. 

\begin{lemm}
  Let $I = [a,b]$ be a compact interval in the real numbers $\RR$, and let $\gamma : I \to \KC(X)$ be a smooth path. The length of the path $\gamma$  satisfies
$$
  L(\gamma) \geq
  \frac{1}{\sqrt n}
  \left| \log \Vol(X,\gamma(b))
    - \log \Vol(X,\gamma(a))
  \right|.
$$
\end{lemm}

\begin{proof}[Sketch of proof.]
We apply the Cauchy--Schwarz inequality to the scalar product $g(u,\omega)$; this gives
$$
|u \cdot \log \Vol(X,\omega)|^2 = |g(u,\omega)|^2
\leq n g(u,u).
$$
Integrating and applying the triangle inequality then gives the announced estimate.
\end{proof}

Applying the lemma on each interval $I_m$ then gives that
\begin{align*}
  L(\gamma) = \lim\limits_{m \to +\infty} L(\gamma_m) = +\infty.
\end{align*}
Thus the limit class $\lim \gamma(t)$ on the boundary cannot be approached by paths in $\KC$ of finite length.

If the K\"ahler and volume cones of $X$ coincide, then these are the the only classes on the boundary and we are done. If not, then there exists a class $\alpha$ on the boundary of $\KC$ such that $\Vol(X,\alpha) > 0$, but there is a proper complex subspace $Z \subset X$ such that $\Vol(Z,\alpha) = 0$.

As $\alpha$ is on the boundary of the K\"ahler cone, then there
exists a K\"ahler class $\omega$ such that $\gamma(t) := \alpha +
t\omega$ is in the K\"ahler cone for all $t > 0$. The tangent vectors of the path $\gamma$ are $\gamma'(t) = \omega$, and the norm of $\gamma'(t)$ at the point $\gamma(t)$ is
$$
\displaylines{
  h(t) := 
  g(\gamma'(t), \gamma'(t))(\gamma(t)) =
  \left(
    \frac{1}{\Vol(X,\gamma(t))}
    \int_X \omega \wedge \frac{(\alpha + t\omega)^{n-1}}{(n-1)!}
  \right)^2
\hfill\cr\hfill
{}- \frac{1}{\Vol(X,\gamma(t))}
    \int_X \omega^2 \wedge \frac{(\alpha + t\omega)^{n-2}}{(n-2)!}.
}
$$
Each of these integrals, and the function $t \mapsto \Vol(X,\gamma(t))$, is a polynomial in $t$ on some small interval $[0,t_0]$. As $\lim_{t\to 0} \Vol(X,\gamma(t)) > 0$ the function $t \mapsto h(t)$ is continuous and positive on a compact interval, so the integral $L(\gamma)$ of its square root exists and is finite.
\end{proof}

\begin{rema}
  The function $-\log \Vol$ is well defined on the entire connected component $\cc P'$ of the cone of classes of positive volume that contains the K\"ahler cone. Its Hessian $g$ then makes $\cc P'$ into a semi-Riemannian manifold. In view of the surface case, where $g$ extends to an honest Riemannian metric on all of $\cc P'$, it seems natural to ask if the same happens in general? I expect the answer to be ``no'', but it is surprisingly hard to construct a counterexample to the question.
\end{rema}

\section{The relative K\"ahler cone}
\label{sefi}

A complex $(1,1)$-class $a$ on a compact manifold $X$ will be called a \em complexified K\"ahler class \em if its imaginary part $\Im a$ is a K\"ahler class. We denote the set of complexified K\"ahler classes on $X$ by $\KCComp(X)$. It is a convex open cone in the finite-dimensional vector space $H^{1,1}(X,\CC)$.

Let $\pi : \cc X \to S$ be a family of compact K\"ahler manifolds over a smooth base $S$. Recall that there is a holomorphic vector bundle $E^{1,1} \to S$ whose fibers are $E^{1,1}_s = H^{1,1}(X_s,\CC)$. The {\em complexified relative K\"ahler cone} of a family $\pi : \cc X \to   S$ is the subset $\cc K$ of $p : E^{1,1} \to S$ that consists of the complexified K\"ahler cones of each manifold $X_s$.

\begin{prop}
  The relative K\"ahler cone $\cc K$ is open in the total space of the vector bundle $E^{1,1}$.
\end{prop}

\begin{proof}[Sketch of proof.]
We adapt the proof of Kodaira--Spencer \cite{KodairaSpencerIII} of the fact that the K\"ahler condition is open in families. Given a point $(a_0,s_0)$ in $\cc K$, we find a relative K\"ahler metric that interpolates that point. The metrics thus obtained on each manifold in the family permit us to identify cohomology classes with harmonic forms on each manifold.

After restricting to the inverse image of a relatively compact neighborhood of the point $s_0$, we note that the unit ball fibration in $T_{\cc X/S}$ is compact over the closure of that neighborhood. Since positivity of forms can be tested on that fibration, we obtain an open ball in $\cc K$ around $(a_0,s_0)$ that is contained in $E^{1,1}$.
\end{proof}

The proposition entails that the complexified relative K\"ahler cone is a complex manifold. It is equipped with a surjective submersion $p : \cc K \to S$ inhereted from its ambient vector bundle, but is not necessarily locally trivial since the K\"ahler cone may vary in families \cite{DemaillyPaun}.

Consider the sheaf $E := \cc R^{2n} \pi_* \CC \otimes_\CC \cc O_S$ over the space $S$. Since the manifolds of our family are compact and K\"ahler and the base $S$ is smooth, $E$ is a holomorphic vector bundle. Its fiber over a point $s$ is $E_s = H^{n,n}(X_s,\CC)$, so $E$ is a holomorphic line bundle.

We now pull this line bundle back to the total space of the relative K\"ahler cone $\cc K$. Then we can define a smooth hermitian metric $g$ on the pullback $p^*E$: if $\mu$ and $\nu$ are local holomorphic sections of $p^*E$, we write
$$
\mu = \frac{u}{\Vol(X,\omega)} \frac{\omega^n}{n!}
\quad\hbox{and}\quad
\nu = \frac{v}{\Vol(X,\omega)} \frac{\omega^n}{n!},
$$
where $u$ and $v$ are holomorphic functions (see the proof of Propostion~\ref{prop:kahlermetric}). We then set
$$
  g(\mu,\overline \nu)_{(a,s)} := 
  u(s) \overline v(s) \Vol(X,\omega).
$$

\begin{prop}
\label{prop:kahlermetric}
  The curvature form of $g$ is
  \begin{align*}
    \frac{i}{2\pi} \Theta_{E,g} = 
    i \partial \bar \partial \, \log\mathop{\rm Vol}(X,\omega).
  \end{align*}
In particular, the restriction of the curvature form to a fiber of $\cc K$ is the negative of the metric on the complexified K\"ahler cone of each manifold $X_s$.
\end{prop}

\begin{proof}
  We claim that the section $\tau(a,s) = (\omega^n/n!) / \Vol(X,\omega)$ of $p^*E$, where $\omega = \Im a$, is holomorphic. To verify the claim, we first note that the section $\tau$ is constant on the fibers of $p : \KCC \to S$, as follows from a few simple calculations. Next note that the section $\tau$ satisfies
\begin{align*}
  \int_{X_s} \tau(a,s) = \frac{1}{\Vol(X_s,\Omega)}
  \int_{X_s} \dV_\Omega = 1
\end{align*}
at all points $(a,s)$ of the space $\cc K$. It is thus dual to the fundamental class of each manifold $X_s$, so it is parallel with respect to the pullback of the Gauss--Manin connection on $E$ to $\cc K$, and thus holomorphic.

Since $\tau$ is a nowhere zero holomorphic section of the line bundle $E$, the curvature form of $g$ is $-i\partial \bar \partial \log |\tau|^2_g$. If we pick a K\"ahler metric $\Omega$ in the class $\omega$, then $|\dV_\Omega|_\Omega = 1$. Thus $|\tau|^2_g = 1/\Vol(X,\omega)$, which implies the result.

Once we restrict to a fiber $K(X_s)$ we only differentiate the function $\log \Vol$ with respect to $(1,1)$-classes on $X_s$. This gives the negative of the K\"ahler metric on the complexified K\"ahler cone of $X_s$.
\end{proof}

  The proposition shows that the curvature form of the hermitian metric $g$ is positive-definite on $T_{\cc K/S}$. It is natural to ask if it is semipositive on the entire space? A more detailed analysis of the variation of the function $\log \Vol$ in horizontal directions is needed to answer this question and we hope to undertake one soon.

\bibliographystyle{alpha}
\bibliography{main}

\end{document}